\documentclass{article}
\usepackage[utf8]{inputenc}

\title{Characteristic Currents on Cohesive Modules}
\author{Zhaobo Han}
\usepackage{url}
\date{April 2024}
\usepackage{amssymb}
\usepackage{amsmath}
\usepackage{amsthm}
\usepackage{bbm}
\usepackage{hyperref}
\usepackage{setspace}
\usepackage{mathrsfs}
\newtheorem{theorem}{Theorem}[section]
\newtheorem{corollary}{Corollary}[theorem]
\newtheorem{lemma}[theorem]{Lemma}
\newtheorem{definition}{Definition}[section]

\newtheorem{question}{Question}[section]
\newtheorem*{remark}{Remark}
\newtheorem{proposition}{Proposition}[section]

\usepackage[all]{xy}
\usepackage{tikz-cd}
\usepackage{mathtools}
\begin{document}

\maketitle
\begin{abstract}
Let $\mathcal{F}$ be a coherent sheaf on a complex variety $X$ that has a locally free resolution $E^{\bullet}$. In \cite{LW}, the authors constructed a pseudomeromorphic current whose support is contained in $supp(E^{\bullet})$ that represents products of Chern classes of $\mathcal{F}.$ In this paper, we show that their construction works for general de-Rham characteristic classes and then generalize it to represent products (in de-Rham cohomology) of characteristic forms of cohesive modules defined by Block\cite{JB1}. Finally, we state a corollary to a transgression result in \cite{HQ} that show that it is sufficient to only use the degree-$0$ and degree-$1$ parts of the superconnection to construct currents\cite{Bis}\cite{BGS} that represent characteristic forms of cohesive modules in the Bott-Chern cohomology.
\end{abstract}
\section{Introduction}
While coherent sheaves on any quasi-projective scheme over a Noetherian affine scheme admits a locally free resolution \cite[Example 6.5.1]{RH}, this is not generally true. An example is certain coherent sheaves on $Spec(K[x]/(x^2)$ for any field $K.$\cite[Example 4.18]{EF}. To circumvent such issue, in \cite[Definition 2.3.2]{JB1}\cite{JB2}, Block introduced the differential-graded category $\mathcal{P}_{\mathcal{D}}$ of \textbf{Cohesive Modules} over the differential-graded algebra (dga) $\mathcal{D}=(\mathcal{A}^{\bullet}(X), d, 0)=(\mathcal{A}^{\bullet, 0}(X), \overline{\partial}, 0)$ the Dolbeault dga of a complex manifold $X$ (and also over general curved dga's) and studied their properties. It has the important property that
\begin{theorem}\cite[Theorem 4.1.3]{JB1}
Let $X$ be a compact complex manifold, and $D_{coh}^{b}(X)$ be the bounded derived category of complexes $\mathcal{O}_X$-sheaves with coherent cohomology. Then the homotopy category $Ho(\mathcal{P}_{\mathcal{D}}),$ whose objects are exactly those of $\mathcal{P}_{\mathcal{D}}$ and whose morphisms $Ho(\mathcal{P}_{\mathcal{D}})(x, y)=H^0(\mathcal{P}_{\mathcal{D}}(x, y)),$ is equivalent to $D_{coh}^b(X).$
\end{theorem}
Thus many statements about coherent sheaves admitting a locally free resolution can be translated into more general statements about cohesive modules.
\vspace{4mm}
Our main result in this paper (Theorem 5.1) generalizes Theorem 5.1 in \cite{LW}, which constructs, for a coherent sheaf $\mathcal{F}$ that admits a locally free resolution on a complex manifold $X$, a current having the same support as a resolution for $\mathcal{F}$ that represents products of Chern classes/Chern forms of $\mathcal{F},$ to a more general class of characteristic forms on cohesive modules which define classes in both the de-Rham cohomology and Bott-Chern cohomology of the manifold $X.$ At the end, we ask some questions about transgression formulae for superconnections and superconnection currents that represent certain Bott-Chern characteristic forms.
\section{Acknowledgements}
I would like to sincerely express my gratitude towards Prof.Jonathan Block, whose patience, guidance, insight, and encouragement helped me overcome many difficulties during both the learning process and the research process. I would also like to thank him for teaching an awesome Algebraic Topology sequence, where I learned many tools necessary for this project.\vspace{4mm}
Special thanks to Prof.Tony Pantev for teaching an awesome Complex Algebraic Geometry sequence, and especially for completing the course despite his health conditions. This project would have been impossible without what I learned from his lectures. I also sincerely thank Prof.Ron Donagi for serving on my thesis committee and for holding the algebraic geometry seminar, which exposed me to cool topics that furthered my interest in algebraic geometry.
\vspace{4mm}
I would like to also extend my gratitude to other faculty members and graduate students who offered me advice throughout the project, and with whom I had fruitful discussions with: Prof.Ted Chinburg, Prof.Ryan Hynd, Fangji Liu, Tianyue Liu, Xingyu Meng, Zixuan Qu, Zhecheng Wu, Shengjing Xu and David Zhu.
\section{Cohesive Modules and Unitary Connections}
\subsection{The \texorpdfstring{$\overline{\partial}$-}{}superconnection}
Let $X$ be a complex manifold, and $\mathcal{D}=(\mathcal{A}^{0, \bullet}(X), \overline{\partial})$ be its Dolbeault differential graded algebra, we can define the dg-category $\mathcal{P}_{\mathcal{D}}$ of $\mathcal{D}-$cohesive modules as follows\cite{HQ}:
the objects are $E=(E^{\bullet}, \mathbb{E}^{''}).$
Here, $E^{\bullet}=\bigoplus_{k=0}^NE_k,$ with each $E^k$ a finite dimensional complex vector bundle over $X.$ 
\begin{proposition}
Here are some basic facts about the sheaves of $E^{\bullet}-$ and $End(E^{\bullet})-$valued differential forms
\begin{enumerate}
    \item $\mathcal{A}^{0, \bullet}(X, E^{\bullet})\cong\mathcal{A}^{0, \bullet}(X)\otimes_{\mathcal{A}^{0}(X)}\mathcal{A}^0(X, E^{\bullet})$ and $\mathcal{A}^{\bullet, 0}(X, E^{\bullet})\cong$ \newline $\mathcal{A}^{\bullet, 0}(X)\otimes_{\mathcal{A}^{0}(X)}\mathcal{A}^0(X, E^{\bullet}).$ Therefore $\mathcal{A}^{\bullet}(X, E^{\bullet})\cong \mathcal{A}^{\bullet}(X)\otimes_{\mathcal{A}^{0}(X)}$ \newline $\mathcal{A}^0(X, E^{\bullet}).$
    \item Same can be said if we replace $E^{\bullet}$ by $End_{\mathbb{C}}(E^{\bullet})$ in (1).
    \item Therefore $End_{\mathcal{O}_X}(\mathcal{A}^{\bullet}(X, E^{\bullet}))\cong \mathcal{A}^{\bullet}(X, End_{\mathbb{C}}(E^{\bullet}))$ as $\mathcal{O}_X-$modules.
\end{enumerate} 
\end{proposition}
Now let $\mathbb{E}^{''}:\mathcal{A}^{0, \bullet}(X, E^{\bullet})\rightarrow\mathcal{A}^{0, \bullet}(X, E^{\bullet})$ be $\mathcal{O}_X$-linear of total degree$-1$ and satisfy the following:
\begin{enumerate}
    \item $\mathbb{E}^{''}\circ\mathbb{E}^{''}=0;$ i.e. $\mathbb{E}^{''}$ is flat.
    \item The $\overline{\partial}-$Leibniz formula $\forall s\in\mathcal{A}^0(X, E^{\bullet}), \forall\omega\in\mathcal{A}^{0, \bullet}(X),$
\begin{equation}
\mathbb{E}^{''}(s\otimes\omega):=\mathbb{E}^{''}(s)\otimes\omega+(-1)^{\deg(\omega)}s\otimes\overline{\partial}(\omega)
\end{equation}
\end{enumerate}
The meaning of total degree$-1$ is that $\forall p, q\in\mathbb{N}$
\[
\mathbb{E}^{''}(\mathcal{A}^{0, p}(X, E^q))\subseteq\bigoplus_{k\geq\max\{-p, -q+1\}}\mathcal{A}^{p+k}(X, E^{q-k+1})\Rightarrow\mathbb{E}^{''}=\bigoplus_{k\in\mathbb{Z}}\mathbb{E}^{''}_k,
\]
with $\mathbb{E}_k^{''}=0, \forall k<\max\{-p, -q+1\}.$ Note that by definition of $\mathbb{E}^{''}$ and degree, we know that $\mathbb{E}^{''}_k$ is $\mathcal{A}^{\bullet}(X)-$linear $\forall k\neq 1.$ Now consider $\forall 0\leq k\leq n,$ we have $\mathbb{E}^{''}(\mathbb{E}^{''}|_{\mathcal{A}^0(X, E^k)})=0,$ so its projection onto $\mathcal{A}^0(X, E^{k+2})$ is also $0.$ Therefore $\mathbb{E}^{''}_0\circ\mathbb{E}^{''}_0(\mathcal{A
}^0(X, E^k))=0$, and we know that
\[
\begin{tikzcd}
0\rar& \mathcal{A}^0(X, E^0)\rar["\mathbb{E}^{''}_0"]& \mathcal{A}^0(X, E^1)\rar["\mathbb{E}^{''}_0"]&\cdots\rar["\mathbb{E}^{''}_0"]& \mathcal{A}^0(X, E^N)\rar& 0
\end{tikzcd}
\]
is a complex of coherent sheaves.
\subsection{Extended Hermitian Metric and the d-connection}
Let $h$ be a Hermitian metric on $E^{\bullet}.$ Then, using Proposition 2.1.(1), we can extend $h$ to $\mathcal{A}^{\bullet}(X, E^{\bullet})$ via
\[
h(\alpha\otimes f, \beta\otimes g)=\overline{\alpha}\wedge h(f, g)\wedge\beta, \hspace{5mm}\forall\alpha, \beta\in\mathcal{A}^{\bullet}(X), \hspace{5mm}\forall f, g\in\mathcal{A}^0(X, E^{\bullet})
\]
We will write a cohesive module as $(E^{\bullet}, \mathbb{E}^{''}, h)$ to emphasize the dependence of various properties/constructions on the Hermitian metric. Now we state an important structure theorem:
\begin{theorem}\cite{HQ} For a Hermitian cohesive module $(E^{\bullet}, \mathbb{E}^{''}, h),$ there is a unique $\mathbb{E}^{'}: \mathcal{A}^{\bullet, 0}(X, E^{\bullet})\rightarrow\bigcup_{q\in\mathbb{Z}}\mathcal{A}^{\bullet+q+1, 0}(X, E^{\bullet-q})$ satisfying the following:
\begin{enumerate}
\item $\mathbb{E}^{'}$ is a $\partial-$superconnection, i.e.\footnote{Note that $\mathcal{A}^{\bullet, 0}(X, E^{\bullet})\cong\mathcal{A}^{\bullet, 0}(X)\otimes_{\mathcal{A}^0(X)}\mathcal{A}^0(X, E^{\bullet})$} $\forall\alpha\otimes f$, with $\alpha\in\mathcal{A}^{\bullet}(X), f\in\mathcal{A}^0(X, E^{\bullet}),$ we have $ \mathbb{E}^{'}(\alpha\otimes f)=\mathbb{E}^{'}(\alpha)\otimes f+(-1)^{deg(\alpha)}\alpha\otimes(\partial f).$
\item Extending $\mathbb{E}^{''}$ and $\mathbb{E}^{'}$ to $\mathcal{A}^{\bullet, \bullet}(X, E^{\bullet})$ linearly (considering that $\mathcal{A}^{\bullet, \bullet}(X, E^{\bullet})$ $\cong\mathcal{A}^{0, \bullet}(X)\otimes_{\mathcal{A}^0(X)}\mathcal{A}^{\bullet, 0}(X)\otimes_{\mathcal{A}^0(X)}\mathcal{A}^0(X, E^{\bullet})$), and we have $\mathbb{E}=\mathbb{E}^{''}+\mathbb{E}^{'}$ is a $d-$superconnection.
\item $\mathbb{E}$ is $h-$unitary, i.e. $\forall s, t\in\mathcal{A}^{\bullet}(X, E^{\bullet}),$ we have $(-1)^{deg(s)}d(h(s, t))=-h(\mathbb{E}(s), t)+h(s, \mathbb{E}(t)).$
\item Writing $\mathbb{E}^{'}=\bigoplus_{q\in\mathbb{Z}}\mathbb{E}^{'}_q, \nabla=\mathbb{E}^{''}_1+\mathbb{E}^{'}_1: \mathcal{A}^{\bullet, \bullet}(X, E^{\bullet})\rightarrow\mathcal{A}^{\bullet+1, \bullet+1}(X, E^{\bullet})$ is a unitary $d-$connection.
\end{enumerate}
\end{theorem}
From the last statement we see that $\nabla$ restricts to connections on each $E^k, 0\leq k\leq N.$ Therefore, writing $\nabla_k=\nabla|_{E_k}$ and $\nabla=\bigoplus_{k=0}^N\nabla_k$, we know that $(E^k, \nabla_k)$ is a vector bundle with a $d-$connection. Note that $\nabla$ induces a connection $\nabla^{End}$ on $End(E^{\bullet})$ by
\[
\nabla^{End}: \mathcal{A}^{\bullet}(X, End(E^{\bullet}))\rightarrow\mathcal{A}^{\bullet+1}(X, End(E^{\bullet})), \hspace{5mm}\phi\mapsto\nabla\circ\phi-(-1)^{deg(\phi)}\phi\circ\nabla.
\]
\subsection{Construction of a Compatible Unitary Connection}
\begin{definition}\cite[Section 4]{LW}
Connections $\{\Theta_k\}_{0\leq k\leq N}$ on $\{E_k\}$ are \textbf{compatible} with the complex
\begin{tikzcd} 0\rar& E_0\rar["\phi_0"]& E_1\rar["\phi_1"]&\cdots\rar["\phi_{N-1}"]&E_N\rar&0
\end{tikzcd}
if $\Theta_{k+1}\circ\phi_k=-\phi_k\circ\Theta_k\Leftrightarrow\Theta_{k+1}\circ\phi_k+\phi_k\circ\Theta_k=0.$
\end{definition}
The reason why compatibility is important will be seen in the next section. In the case of cohesive modules, the $\nabla_k$'s might not be compatible with $\mathbb{E}^{''}_0.$
\vspace{4mm}
Let $Z_i$ be the set where $\mathbb{E}^{''}_0$ is not exact, and let $Z=\bigcup_{i=0}^NZ_i.$ We will call $Z$ the support of the cohesive module. Let $\chi: \mathbb{R}\rightarrow [0, 1]$ be a smooth characteristic function such that $\chi\equiv 0$ on $(-\infty, 1-\delta)$ and $\chi\equiv 1$ on $(1+\delta, \infty),$ for some arbitrarily small positive $\delta.$
\\\\
Now for a positive $\epsilon>0,$ if $Z=\emptyset,$ define $\chi_{\epsilon}\equiv 1$ on $X.$ Otherwise, define $F=\boxtimes_{k=1}^NF_k=\boxtimes_{k=1}^N det(\mathbb{E}_0^{''}|_{E_{k-1}})^{\bigwedge rank(E_{k-1})}$ on $X$\cite[Section 2]{AW}, which is a section to the coherent sheaf $\mathcal{F}=\boxtimes_{k=1}^N\left(\bigwedge^{rank(E_{k-1})}E_{k-1}^*\otimes\bigwedge^{rank(E_{k-1})}E_k\right)$. Then it is clear that $Z=\{F=0\}=\bigcup_{k=1}^N\{F_k=0\}$. If it is impossible to find an $F$ that is generically nonvanishing such that $Z\subseteq\{F=0\}$, define $\chi_{\epsilon}\equiv 1.$ Otherwise define $\chi_{\epsilon}(x)=\chi\left(\frac{|F(x)|^2}{\epsilon}\right), \forall x\in X.$ In this case $\chi_{\epsilon}\equiv 1$ except on a small neighborhood of $Z$ when $\epsilon$ is small. (Since $F$ is generically nonvanishing, we can modify $F$ such that $F\equiv 1$ except in a small neighborhood of $Z.$) Now we need another concept before constructing the compatible connections: the minimal inverse.
\begin{definition} \cite{LW, AW}
For $0\leq k\leq N-1,$ write $\mathcal{A}^0(X, E_{k+1})=\mathbb{E}^{''}_0(\mathcal{A}^0(X, E_k))$ $\oplus F_{k+1}.$ Define the \textbf{minimal inverse} $\sigma_k: E_{k+1}\rightarrow E_k$, a morphism between vector bundles, by the following conditions: if $e\in\mathbb{E}^{''}_0(E_k)$, then $\sigma_k(e)\equiv n,$ where $\mathbb{E}^{''}_0(n)=e,$ and $n$ has pointwise the minimal $h-$norm among all such vectors. If under $h, e\perp\mathbb{E}^{''}_0(E_k),$ then $\sigma_k(e)\equiv 0.$ It then follows that $\mathbb{E}^{''}_0\circ\sigma_k\circ\mathbb{E}^{''}_0=\mathbb{E}^{''}_0.$
\end{definition}
\begin{remark} Here are some properties of $\sigma_k$
\begin{enumerate}
\item The minimality of $\sigma_k(e)$ is equivalent to stating that $n\perp Ker(\mathbb{E}^{''}_0|_{E_k}),$ since any complement of $Ker(\mathbb{E}^{''}_0|_{E_k})$ injects onto $\mathbb{E}^{''}_0(E_k)$ under $\mathbb{E}^{''}_0.$\cite{MA}
\item From Remark (1), we know that $Im(\sigma_k)\perp Ker(\mathbb{E}^{''}_0|_{E_k})\Rightarrow Im(\sigma_k)\perp\mathbb{E}^{''}_0(E_{k-1}),$ since $(\mathbb{E}^{''}_0)^2=0.$ This means that $\sigma_{k-1}\sigma_k=0.$\cite{LW}
\item $\sigma_k$ is smooth on $X\backslash Z_k.$ This is because on $X\backslash Z_k,$ the rank of $\mathbb{E}^{''}_0|_{E_k}$ is constant. It then suffices to show that $\sigma_k|_{\mathbb{E}^{''}_0(E_k)}$ is smooth, since $\sigma_k$ on the orthogonal complement is constant. This follows from a description of $\sigma_k$ in \cite[Section 3]{MA}.
\end{enumerate}
\end{remark}
Now we construct the connections $\nabla_k^{\epsilon}=\nabla_k-\chi_{\epsilon}(\sigma_k\circ\nabla^{End}\circ\mathbb{E}^{''}_0)$ on $E_k$'s, and we write $\nabla^{\epsilon}=\bigoplus_{k=0}^N\nabla_k^{\epsilon}.$ Then we have
\begin{theorem} \cite[Lemma 4.4]{LW}
For any $\epsilon>0$, the connections $\{\nabla_k^{\epsilon}\}_{0\leq k\leq N}$ are compatible with $\mathbb{E}^{''}_0$ exactly where $\chi_{\epsilon}\equiv 1.$
\end{theorem}
\section{Characteristic Class in de-Rham Cohomology}
\subsection{Characteristic Forms of \texorpdfstring{$(E^{\bullet}, \mathbb{E}^{''}, h)$}{}}
Let $(E^{\bullet}, \mathbb{E}^{''}, h)$ be a Hermitian cohesive module. Define the curvature $R_h=\mathbb{E}^2=\frac 12[\mathbb{E}, \mathbb{E}]=[\mathbb{E}^{'}, \mathbb{E}^{''}].$ 
\begin{remark}
Noting that $\mathbb{E}\in End(\mathcal{A}^{\bullet}(X, E^{\bullet}))\cong\mathcal{A}^{\bullet}(X, End(E^{\bullet}))\cong$ \newline$\mathcal{A}^{\bullet}(X)\otimes_{\mathcal{A}^0(X)}\mathcal{A}^0(X, End(E^{\bullet})),$ so if we write $\mathbb{E}=\alpha\otimes f,$ then we have $R_h=(\alpha\wedge\alpha)\otimes (f\circ f).$
\end{remark}
Then, following Quillen's notion of the supertrace\cite{DQ}, for a fixed convergent complex power series $f(T),$ we define its \textbf{characteristic form} to be $Tr_s(f(R_h))$, where $Tr_s:\mathcal{A}^{\bullet}(X, End(E^{\bullet}))\rightarrow\mathcal{A}^{\bullet}(X)$ is defined as follows \cite{DQ, AW}: letting $E^{\bullet}=E^{+}\oplus E^{-},$ with $E^{+}=\bigoplus_{2|k}E_k, E^{-}=\bigoplus_{2\not|k}E_k,$ we define $Tr_s: End(E^{\bullet})\rightarrow\mathbb{C}, X\mapsto tr(\epsilon X),$ where for $e\in E^{+}, \epsilon X(e)=X(e)$, and for $e\in E^{-}, \epsilon X(e)=-X(e).$ Now extend $Tr_s$ linearly to $\mathcal{A}^{\bullet}(X, End(E^{\bullet})).$ Then we have the following facts
\begin{theorem}\cite{HQ}[Corollary 2.26]
Characteristic forms are closed, so they define classes in $H^{\bullet}_{dR}(X, \mathbb{C}).$ These classes are well-defined by Serre's Vanishing Theorem. We then have $[Tr_s(f(R_h))]=[Tr_s(f(\Theta_{\nabla}))]$\footnote{This $\nabla$ was defined in Theorem 2.1.(4).} in $H^{\bullet}_{dR}(X, \mathbb{C}),$ where $\Theta_{\nabla}=\nabla^2$ is the curvature form associated to $\nabla.$ 
\end{theorem}
\subsection{Characteristic Forms of Exact Chain Complexes}
We first show the following claim establishing a more explicit relation between the characteristic form and the curvature form. Let $(E, \nabla)$ be a vector bundle on $X.$ Denote $tr:\mathcal{A}^{\bullet}(X, End(E))\rightarrow\mathcal{A}^{\bullet}(X)$ the extension of the trace function on $\mathcal{A}^0(X, End(E)).$ Then we have
\begin{proposition}
$[tr(f(\Theta_{\nabla}))]\in H^*_{dR}(X)$ is a polynomial in the Chern classes of $E.$ Specifically, it is a symmetric polynomial in $[\Theta_{\nabla}].$ It is also a sum of homogeneous polynomimals in $[\Theta_{\nabla}].$
\end{proposition}
\begin{proof}
Let $X, Y$ be two complex algebraic varieties, and $(E, \Delta)$ a vector bundle of rank $k$ on $X.$ We show that $(E, \Delta)\mapsto [tr(f(\Theta_{\Delta}))]$ is a natural transformation from $\mathsf{Vect}_k(-; \mathbb{C})$ to $H^*(-).$ For a morphism $\phi: Y\rightarrow X$, let $(\phi^*E, \phi^*\Delta)$be the pullback vector bundle on $Y$ with the pullback connection which is functorial (as defined in \cite[Theorem 3.6(a)]{ROW}), we know that $\Theta_{\phi^*\Delta}=\phi^*\Theta_{\Delta}.$ Now it suffices to show that $\forall i\in\mathbb{N}, \phi^*[tr(\Theta_{\Delta}^i)]=[tr((\phi^*\Theta_{\Delta})^i)].$ Recall the splitting principle
\begin{lemma}\cite[Section 21]{BT}
Let $E\rightarrow X$ a $C^{\infty}$ complex vector bundle and $p:\mathbb{P}(E)\rightarrow M$ be the projection map. Then $p^*(E)\rightarrow\mathbb{P}(E)$ splits into a direct sum of line bundles and $p^*:H^*(X)\rightarrow H^*(\mathbb{P}(E))$ is an embedding.
\end{lemma}
Using the lemma and the fact that $\phi^*\mathbb{P}(E)=\mathbb{P}(\phi^*(E))$, and considering the commutative diagram
\[\begin{tikzcd}
	{L_1\oplus\cdots\oplus L_n\rightarrow\mathbb{P}(E)} && {\phi^*(L_1)\oplus\cdots\oplus\phi^*(L_n)\rightarrow\mathbb{P}(\phi^*(E))}\\
	\\
	{E\rightarrow X} && {\phi^*(E)\rightarrow Y}
	\arrow["{\phi^*}"', from=3-1, to=3-3]
	\arrow["{p_E^*}", from=3-1, to=1-1]
	\arrow["{\phi^*}", from=1-1, to=1-3]
	\arrow["{P^*_{\phi^*(E)}}"', from=3-3, to=1-3]
\end{tikzcd}\]
and then noting that $\mathbb{P}(\phi^*(E))=\phi^*(\mathbb{P}(E)),$ we can reduce to when $E\rightarrow X$ is a line bundle, in which case $\phi^*$ amounts to multiplication by an element in $\mathcal{O}_X$ on both sides.
\\\\
Now recall that every natural transformation $\mathsf{Vect}_k(-, \mathbb{C})\rightarrow H^*_{dR}(-)$ can be expressed as a polynomial in the Chern classes \cite[Proposition 23.11]{BT}, it remains to show that the Chern classes $c_n(E)$ is a polynomial of $[\Theta_{\Delta}].$ This follows directly from \cite[Definition 3.4]{ROW}. 
\end{proof}
\begin{remark}
To show that $[tr(f(\Theta_{\nabla}))]$ is a polynomial, not a series, in the Chern classes, we implicitly used the fact that degrees$\geq 2dim_{\mathbb{C}}(X)$ vanish in $H^*_{dR}(X).$
\end{remark}
Now we recall a Whitney formula \cite[Lemma 4.22]{BB}
\begin{theorem}
Let $\phi$ be a symmetric homogeneous polynomial of degree less than or equal to $dim_{\mathbb{C}}(X),$ then let $\{D_k\}_{0\leq k\leq N}$ be compatible connections on $(E^{\bullet}, \mathbb{E}^{''}).$ If the complex $(E^{\bullet}, \mathbb{E}_0^{''})$ is exact, then in de-Rham cohomology, \newline $[\phi(\Theta_{D_0})]=-\left(\frac{2\pi}i\right)^{deg(\phi)}\cdot[\phi(\sum_{i=1}^N(-1)^iE_i)]$\cite[Equation 4.15]{BB}, which is defined as follows: letting $\sigma_k(D_j)=$ $\left(\frac{2\pi}i\right)^k\cdot c_k(E_j, D_j)$\footnote{Here $\sigma_k$ denotes the elementary symmetric polynomial of degree $k.$}\cite[Equation 3.34]{BB}, and $\phi(D_j)=\tilde{\phi}(\sigma_1(D_j), \cdots, \sigma_{deg(\phi)}(D_j)),$ and writing \newline $\sum_{i=1}^N(-1)^iE_i$ as $\widehat{E^{\bullet}},$ we define $\phi(\widehat{E^{\bullet}})\equiv \tilde{\phi}(c_1(\widehat{E^{\bullet}}), \cdots, c_{deg(\phi)}(\widehat{E^{\bullet}})).$  
\end{theorem}
\begin{remark}
Here, the Chern classes of $\widehat{E^{\bullet}}$ satisfy the following:
\[
1+\sum_{i=1}^{2dim_{\mathbb{C}}(X)}c_i(\widehat{E^{\bullet}})=c(\widehat{E^{\bullet}})\equiv\bigoplus_{k=0}^Nc(E_i)^{(-1)^i}
\]
\end{remark}
Since $\phi$ is homogeneous, we know that $[\phi(\Theta_{D_0})]=\left(\frac{2\pi}i\right)^{deg(\phi)}$\\$[\tilde{\phi}(c_1(E_0, D_0), \cdots, c_{deg(\phi)}(E_0, D_0)].$ Therefore From this and Proposition 3.1, we automatically have
\begin{corollary}
On $X\backslash Z,$
\[
Tr_s(f(\Theta_{\nabla}))=\bigoplus_{k=0}^NTr_s(f(\Theta_{\nabla_k}))=\sum_{k=0}^N(-1)^ktr(\Theta_{\nabla_k})=0.
\]
\end{corollary}

\section{Characteristic Currents}
Now we define the characteristic currents on a cohesive module $(E^{\bullet}, \mathbb{E}^{''}, h).$ For $a\in H^*_{dR}(X),$ denote $a_k$ to be the degree-$k$ component of $a.$
\begin{definition}
For $(p_1, \cdots, p_k)\in [0, 2dim_{\mathbb{C}}(X)]^k$ define the \textbf{characteristic current} to be $Tr_{p_1}(E^{\bullet}, \nabla)\wedge\cdots\wedge Tr_{p_k}(E^{\bullet}, \nabla)\equiv\lim_{\epsilon\to 0}[Tr_s(f(\nabla^{\epsilon}))]_{p_1}\wedge\cdots\wedge[Tr_s(f(\nabla^{\epsilon}))]_{p_k}.$
\end{definition}
\begin{remark}
We will explain the meaning of wedge product right after Definition 5.2 (assuming Theorem 5.1 a priori).
\end{remark}
\begin{definition}\cite{AW2}
Let $f$ be a holomorphic function on $X.$ For $a\in\mathbb{N},$ the current $[\frac 1{f^a}]$ is defined as the functional on test forms $\xi\mapsto\lim_{\epsilon\to 0}\int_{|f|>\epsilon}\frac{\xi}{f^a}$ and $\overline{\partial}[\frac 1{f^a}]$ sends test form $\xi$ to $\lim_{\epsilon\to 0}\int_{|f|>\epsilon}\frac{\overline{\partial}(\xi)}{f^a}.$ These are well-defined by\cite[Theorem 7.1]{HL}. Let $\Pi=\Pi_1\circ\Pi_2\circ\cdots\circ\Pi_r$ be a sequence of resolutions of singularities, with $\Pi_i: Y_i\rightarrow Y_{i-1}$ with $Y_0=X.$ Then a current on $X$ is \textbf{pseudomeromorphic} if it can be written as $\sum_{\ell}\Pi_*\tau_{\ell},$ where $\tau_{\ell}$ is a current on some $Y_{\ell}$ of the form $\left(\prod_{i=1}^k[\frac 1{f^{a_i}}]\right)\overline{\partial}[\frac 1{f^{b_1}}]\wedge\cdots\wedge\overline{\partial}[\frac 1{f^{b_m}}]$ for some holomorphic $f$ on $Y_{\ell}.$
\end{definition}
Here, we need to define a notion of "wedge product" of currents. This is given by the \textbf{Coleff-Herrera product}\cite[Theorem 1.7.2]{CH}\cite[Theorem 2]{LK}. Call $(\epsilon_1, \cdots, \epsilon_p)\rightarrow (0, \cdots, 0)$ along an \textit{admissible path} if $\forall k\in\mathbb{N}$ and $\forall j\geq 2, \frac{\epsilon_{j-1}}{\epsilon_j^k}\to 0.$ in this case we write $\epsilon_1\ll\cdots\ll\epsilon_p\to 0.$ Then for $f_1, \cdots, f_p$ holomorphic, we define
\[
\overline{\partial}[\frac 1{f_1}]\wedge\cdots\wedge\overline{\partial}[\frac 1{f_p}]=\lim_{\epsilon_1\ll\cdots\ll\epsilon_p\to 0}\frac{\overline{\partial}\chi(|f_1|^2/\epsilon_1)}{f_1}\wedge\cdots\wedge\frac{\overline{\partial}\chi(|f_p|^2/\epsilon_p)}{f_p},
\]
where for each $(\epsilon_1, \cdots, \epsilon_p)$ and for any test form $\phi$ of bi-degree $(dim_{\mathbb{C}}(X),$ $dim_{\mathbb{C}}(X)-p),$ the expression on the right hand side denotes $\phi\mapsto\int_X\frac{\overline{\partial}\chi(|f_1|^2/\epsilon_1)}{f_1}\wedge\cdots\wedge\frac{\overline{\partial}\chi(|f_p|^2/\epsilon_p)}{f_p}\wedge\phi.$ Then for holomorphic $f_1, \cdots, f_k, g_1, \cdots, g_m$, and for a $(dim_{\mathbb{C}}(X),$ $dim_{\mathbb{C}}(X)-p)-$test form $\phi,$ define $\left(\prod_{i=1}^k[\frac 1{f_i}]\right)\overline{\partial}[\frac 1{g_1}]\wedge\cdots\wedge\overline{\partial}[\frac 1{g_m}](\phi)$ as
\[
\lim_{\substack{\delta_1, \cdots, \delta_k\to 0\\\epsilon_1\ll\cdots\ll\epsilon_m\to 0}}\int_{|f_i|>\delta_i, \forall i}\frac {\bigwedge_{j=1}^m\overline{\partial}\chi\left(\frac{|g_j|^2}{\epsilon_j}\right)\wedge\phi}{\prod_{i=1}^kf_i\prod_{j=1}^mg_j}
\]
\begin{remark}
It follows from Definition 5.2 that pushforwards of pseudomeromorphic currents under resolutions of singularities are still pseudomeromorphic.
\end{remark}
\begin{theorem}
The characteristic current $Tr_{p_1}(E^{\bullet}, \nabla)\wedge\cdots\wedge Tr_{p_k}(E^{\bullet}, \nabla)$ is a well-defined closed pseudomeromorphic current, with support contained in $Z,$ an analytic subvariety of positive codimension, that represents $Tr_s(f(\Theta_{\nabla}))_{p_1}\wedge\cdots\wedge Tr_s(f(\Theta_{\nabla}))_{p_k}$ for any $k-$tuple $(p_1, \cdots, p_k).$
\end{theorem}
\begin{remark}
Denote $(\mathbb{D}^*(X), d)$ the complex of currents on $M.$ Here $\mathbb{D}^q(X)$ is the dual space to $\Omega^{2dim_{\mathbb{C}}(X)-q}_c(X),$ the vector space of compactly supported smooth forms on $X,$ with the dual topology. Also, the chain map $d:\mathbb{D}^q(X)\rightarrow\mathbb{D}^{q+1}(X)$ is given by $(dT)(\phi)=(-1)^{q+1}T(d\phi).$ Then we have $H^*_{dR}(X)\overset{\cong}\longrightarrow H^*(\mathbb{D}^*(X), d).$\cite[Chapter 3, Section 1]{GH}
\end{remark}
\begin{proof}
By Proposition 4.1 and linearity, it suffices to consider the case where $Tr_s(f(\Theta_{\nabla}))$ is a monomial in the Chern classes. Then from \cite[Lemma 2.1]{LW} and \cite[Theorem 5.1]{LW}, it suffices to show that $\forall\ell_1, \ell_2$ we have
\[
\lim_{\epsilon\to 0}c_{\ell_1}(E^{\bullet}, \nabla^{\epsilon})\wedge\lim_{\delta\to 0}c_{\ell_2}(E^{\bullet}, \nabla^{\delta})=\lim_{\epsilon\to 0}c_{\ell_1}(E^{\bullet}, \nabla^{\epsilon})\wedge c_{\ell_2}(E^{\bullet}, \nabla^{\epsilon})
\]
and then proceed inductively (for wedge products of more Chern classes). Since $[tr(f(\Theta_{\nabla}))]$ is a characteristic class, it does not depend on the choice of connection. As in the proof of \cite[Theorem 5.1]{LW}, for any $\epsilon$ and $\delta,$ we can write $c_{\ell_1}(E^{\bullet}, \epsilon)$ as $A_1+\sum_{j\geq 1}\chi_{\epsilon}^jB_j+\sum_{j\geq 1}\chi_{\epsilon}^{j-1}\wedge d\chi_{\epsilon}\wedge B_j'$ and $c_{\ell_2}(E^{\bullet}, \delta)$ as $A_2+\sum_{j\geq 1}\chi_{\delta}^jC_j+\sum_{j\geq 1}\chi_{\delta}^{j-1}\wedge d\chi_{\delta}\wedge C_j',$ where $A_1, A_2, B_j, C_j, B_j', C_j'$ are independent of $\epsilon$ and $\delta,$ with $A_1, A_2$ smooth and $B_j, C_j, B_j', C_j'$ polynomials in the entries of the minimal inverses $\sigma_k$ (cf. Definition 2.2), $D_{End(E^{\bullet})}\mathbb{E}_0^{''},$ and $\theta_k$, which are the matrices representing $\nabla_k$ ($0\leq k\leq N$). Also, $\chi_{\epsilon}=\chi(\frac{|F|^2}{\epsilon})$ as defined in Section 3.3. Therefore it suffices to consider where $\chi_{\epsilon}\not\equiv 1$ for $\epsilon$ small enough and show that for any $i, j$ and any $s, t$ products of entries of $\sigma_k, D_{End(E^{\bullet})}\mathbb{E}_0^{''}$ and $\theta_k$, we have $\chi_{\epsilon},
\lim_{\epsilon\to 0}\chi_{\epsilon}^is\wedge\lim_{\delta\to 0}\chi_{\delta}^jt=\lim_{\epsilon}\chi_{\epsilon}^{i+j}s\wedge t,$ and also $\lim_{\epsilon\to 0}\chi_{\epsilon}^id\chi_{\epsilon}\wedge s\wedge\lim_{\delta\to 0}\chi_{\delta}^jd\chi_{\delta}\wedge t=\lim_{\epsilon\to 0}\chi_{\epsilon}^{i+j}d\chi_{\epsilon}\wedge s\wedge d\chi_{\epsilon}\wedge t.$
\\\\
We use similar ideas to \cite[Lemma 2.1]{LW}. By resolution of singularities \cite[Theorem 3.36]{JK}, since in Section 2.3 we found a section $F=\boxtimes_{j=0}^NF_j$ to a coherent sheaf (actually a vector bundle) $\mathcal{F}$ such that $Z\subseteq Z(F),$ which implies that $(F)\cdot\mathcal{O}_X$\footnote{Here $(F)$ is the ideal of $\mathcal{O}_X$ generated by $(F).$} is not invertible on $Z(F),$ we know that there exists a (composition of) birational and projective modification(s) $\pi:Y\rightarrow X$ such that $\pi|_{Y-\pi^{-1}(Z(F))}: Y-\pi^{-1}(Z(F))\rightarrow X-Z(F)$ is an isomorphism, and the coherent sheaf of ideals on $Y$ generated by pullbacks of local sections to $(F)\cdot\mathcal{O}_X,$ which we write as $\pi^{-1}\left((F)\cdot\mathcal{O}_X\right),$ is a monomonial sheaf of ideals.\footnote{Also reference \cite[Note 3.16]{JK} for the equivalent characterizations of monomial ideal sheaves.} Equivalently, this is the subsheaf of $\mathcal{O}_Y$ generated by $\pi^*F=\boxtimes_{j=0}^N\pi^*F_j.$ This means that $\pi^*F_j=F_{j0}F_{j1},$ where $F_{j0}=\prod_{i=1}^nz_i^{c_i}$ is a monomial in local coordinates $\{z_1, \cdots, z_n\}$ and $Z(F_{j0})\subseteq\pi^{-1}(Z(F))$ and $F_{j1}$ is holomorphic and nonvanishing. Then on $Y$, we have the local formula
\begin{equation}
    \pi^*\sigma_k=\frac 1{F_{k0}}\phi_k, \forall 0\leq k\leq N
\end{equation} 
with $\phi_k$ smooth everywhere on $Y.$ (cf.the definition of $\sigma_k$ in Definition 2.2. This can also be found in \cite[Section 2]{AW}) Note that $D_{End(E^{\bullet})}\mathbb{E}_0^{''}$ and $\theta_k$ are everywhere smooth. Now by \cite[Equation 2.2]{LW}, writing $\pi^*s=\frac 1{\psi_1}\tilde{s}, \pi^*t=\frac 1{\psi_2}\tilde{t}$, with $\psi_1, \psi_2$ products of monomials (thus also monomials) of local coordinates on $Y$ and $\tilde{s}, \tilde{t}$ smooth.\footnote{More accurately $\pi^*s$ might be a sum of such terms, but we can certainly apply a normalization argument to achieve a fraction whose denominator is a monomial in local coordinates.} In view of Equation (2), it suffices to show that for any test $2dim_{\mathbb{C}}(X)$-form $\xi$ on $Y$ (here we note that $\chi\sim\chi^i, \forall i\in\mathbb{N},$ and also $\{\psi_1=0\}\bigcup\{\psi_2=0\}\subseteq\pi^{-1}(Z)=\{\pi^*F=0\},$ using the fact that $\pi$ is an isomorphism on $X-Z(F)$ and $s, t$ are smooth outside of $Z$)
\[
\bigg[\lim_{\epsilon\to 0}\frac{\chi\left(\frac{|\pi^*F|^2}{\epsilon}\right)^{i+j}}{\psi_1\psi_2}\bigg](\xi)=\bigg[\lim_{\epsilon\to 0}\frac{\chi\left(\frac{|\pi^*F|^2}{\epsilon}\right)^i}{\psi_1}\bigg]\bigg[\lim_{\delta\to 0}\frac{\chi\left(\frac{|\pi^*F|^2}{\delta}\right)^j}{\psi_2}\bigg](\xi)\Longleftrightarrow
\]
\begin{equation}
\lim_{\epsilon\to 0}\int_{|\psi_1\psi_2|>\epsilon}\frac{\xi}{\psi_1\psi_2}=\lim_{\substack{\epsilon\to 0\\\delta\to 0}}\int_{\substack{|\psi_1|>\epsilon\\|\psi_2|>\delta}}\frac{\xi}{\psi_1\psi_2}
\end{equation}
since this current does not depend on the choice of characteristic function.\cite[Lemma 2.1]{LW} Then the difference between the two is (noting that $\xi$ is a test form so $||\xi||_{L^{\infty}(X)}$ exists)
\[
\lim_{\substack{\epsilon\to 0\\\delta\to 0}}\int_{|\psi_1|>\epsilon}\frac{(\mathbbm{1}_{|\psi_2|>\epsilon/|\psi_1|}-\mathbbm{1}_{|\psi_2|>\delta})\xi}{\psi_1\psi_2}+\lim_{\epsilon\to 0}\int_{|\psi_1|<\epsilon}\frac{\mathbbm{1}_{|\psi_2|>\epsilon/|\psi_1|}\xi}{\psi_1\psi_2}
\]
\[
\leq\left(\lim_{\substack{\epsilon\to 0\\\delta\to 0}}\int_{\substack{|\psi_1|>\epsilon\\\delta>\epsilon/|\psi_1|\\\epsilon/|\psi_1|<|\psi_2|<\delta}}\left|\frac{\xi}{\psi_1\psi_2}\right|+\int_{|\psi_1\psi_2|<\epsilon}\left|\frac{\xi}{\psi_1\psi_2}\right|\right)+\lim_{\epsilon\to 0}\int_{|\psi_1|<\epsilon}\left|\frac{\xi}{\psi_1}\right|\leq
\]
\[
||\xi||_{L^{\infty}(X)}\left(\lim_{\substack{\epsilon\to 0\\\delta\to 0}}\int_{\substack{|\psi_1|>\epsilon\\\delta>\epsilon/|\psi_1|\\\epsilon/|\psi_1|<|\psi_2|<\delta}}\left|\frac{\mathbbm{1}_{Supp(\xi)}}{\psi_1\psi_2}\right|+\lim_{\epsilon\to 0}\int_{|\psi_1\psi_2|<\epsilon}\left|\frac{\mathbbm{1}_{Supp(\xi)}}{\psi_1\psi_2}\right|+\int_{|\psi_1|<\epsilon}\left|\frac{\mathbbm{1}_{Supp(\xi)}}{\psi_1}\right|\right)
\]
We have
\[
\lim_{\substack{\epsilon\to 0\\\delta\to 0}}\int_{\substack{|\psi_1|>\epsilon\\\delta>\epsilon/|\psi_1|\\\epsilon/|\psi_1|<|\psi_2|<\delta}}\left|\frac{\mathbbm{1}_{Supp(\xi)}}{\psi_1\psi_2}\right|\leq\lim_{\substack{\epsilon\to 0\\\delta\to 0}}\int_{\substack{|\psi_1|>\epsilon\\|\psi_2|<\delta}}\left|\frac{\mathbbm{1}_{Supp(\xi)}}{\psi_1\psi_2}\right|
\]
\[
\leq \lim_{\epsilon\to 0}\sqrt{\int_{|\psi_1|>\epsilon}\frac {\mathbbm{1}_{Supp(\xi)}}{|\psi_1|^2}}\cdot\lim_{\delta\to 0}\sqrt{\int_{|\psi_2|<\delta}\frac {\mathbbm{1}_{Supp(\xi)}}{|\psi_2|^2}}=O\left(\lim_{\delta\to 0}\int_{|\psi_2|<\delta}\frac {\mathbbm{1}_{Supp(\xi)}}{|\psi_2|}\right),
\]
which will follow from Lemma 5.3. To prove that this is $0,$ and also the second and third terms vanish, it suffices to prove the following two lemmas:
\begin{lemma}
Let $\phi$ be a holomorphic function on $M$ whose vanishing locus is a measure-zero set (specifically, a subvariety of positive codimension), and such that $\frac 1{\phi}\in L^1(Supp(\xi)),$ then $\lim_{\epsilon\to 0}\int_{|\phi|<\epsilon}\frac 1{|\phi|}=0.$ Here, $|\cdot|$ is the complex norm.
\end{lemma}
\begin{proof}
Note that $codim(Z(\phi))>0\Rightarrow\mu(Z(\phi))=0,$ where $\mu$ is the pullback (under the coordinate maps) of the Lebesgue measure on $\mathbb{C}^{dim_{\mathbb{C}}(X)}.$ Define
\begin{equation}
    f_n: X\rightarrow [0, +\infty],\quad x\mapsto
    \begin{cases*}
       \frac 1{|\phi(x)|} & $|\phi(x)|<\frac 1n$ \\
       0        & otherwise
    \end{cases*}
  \end{equation}
Then note that $f_i(x)\geq f_j(x), \forall i<j,$ and $\lim_{n\to\infty}f_n(x)\neq 0\Leftrightarrow \phi(x)=0.$ Now since $\int_{Supp(\xi)}|f_1|\leq\int_{Supp(\xi)}\frac 1{|\phi|}<\infty$, then by monotone convergence, we have
\[
\lim_{\epsilon\to 0}\int_{|\phi|<\epsilon}\frac {\mathbbm{1}_{Supp(\xi)}}{|\phi|}=\lim_{n\to\infty}\int_{|\phi|<\frac 1n}\frac {\mathbbm{1}_{Supp(\xi)}}{|\phi|}=\lim_{n\to\infty}\int_{Supp(\xi)}f_n
\]
\[
=\int_{Supp(\xi)}\lim_{n\to\infty}f_n=\int_{Supp(\xi)\cap Z(\phi)}\lim_{n\to\infty}f_n,
\]
which is $0$ since $\mu(Supp(\xi)\cap Z(\phi))\leq\mu(Z(\phi))=0.$
\end{proof}
\begin{lemma}
We can apply a further sequence of blow-ups $\Pi:\tilde{Y}\rightarrow Y$ such that $\frac 1{\Pi^*\psi_1}$ and $\frac 1{\Pi^*\psi_2}\in L^1({Supp(\xi)}).$ Therefore we can assume WLOG that $\frac 1{\psi_1}, \frac 1{\psi_2}\in L^1({Supp(\xi)}).$
\end{lemma}
\begin{proof}
We will show this for $\psi_1$ only. In local coordinates, write $\psi_1=\prod_{i=1}^nz_i^{c_i}.$ Then by Fubini, and writing $r=diam(Supp(\xi))<\infty$ we have $||\psi_1||_{L^1({Supp(\xi)})}\leq\prod_{i=1}^n\int_{|z_i|\leq r}\frac 1{|z_i|^{c_i}}$, which is finite if $c_i$ is smaller than $2dim_{\mathbb{C}}(Y), \forall i.$ Now since $Supp(\xi)\cap Z(\pi^*F)$ is compact, we can cover it by finitely many coordinate neighborhoods, and $|\psi_1\psi_2|$ will have a strictly positive lower bound outside these neighborhoods. Thus the problem reduces to finding $\Pi:\tilde{Y}\rightarrow Y$ such that $\Pi^*\psi_1$ and $\Pi^*\psi_2$ are monomials in local coordinates covering $Supp(\xi)\cap Z(\pi^*F)$ in which the degree of every coordinate does not exceed $2dim_{\mathbb{C}}(Y)-1.$ This will follow directly from \cite[Theorem 3.68]{JK}, which states that if $I\subseteq\mathcal{O}_X$ is an ideal sheaf with $\mathsf{Maxord}_{Supp(\xi)\cap Z(\pi^*F)}(I)\leq m$ for some $m\in\mathbb{N}$, then there is a composition of blow-ups $\Pi=\Pi_r\circ\cdots\circ\Pi_1$ such that $\mathsf{Maxord}_{\Pi^{-1}(Supp(\xi)\cap Z(\pi^*F))}\Pi^{-1}(I)=\mathsf{Maxord}_{\Pi^{-1}(Supp(\xi))\cap Z(\Pi^*\pi^*F)}<m.$ Here, for a point $y\in Y, \mathsf{ord}_y(I)\overset{def}=\max\{r:\mathfrak{m}_y^r\mathcal{O}_{Y, y}\supset I_x\},$ where $\mathfrak{m}_y$ is the maximal ideal of $\mathcal{O}_{Y, y};$ for a subset $Z\subseteq Y,$ define $\mathsf{Maxord}_Z(I)=\sup_{y\in Z}\mathsf{ord}_y(I).$\cite[Definition 3.47]{JK} In our case, we let $I=(\psi_1, \psi_2)\mathcal{O}_Y.$ Then $\mathsf{ord}_y(I)$ is just the maximal order of the two monomials.
\end{proof}
\textit{(Proof of Theorem 5.1--Continued)}\\ Now we show that $\lim_{\epsilon\to 0}\pi^*(\chi_{\epsilon})^id\pi^*\chi_{\epsilon}\wedge \pi^*s\wedge\lim_{\delta\to 0}(\pi^*\chi_{\delta})^jd\pi^*\chi_{\delta}\wedge \pi^*t=\lim_{\epsilon\to 0}(\pi^*\chi_{\epsilon})^{i+j}d\pi^*\chi_{\epsilon}\wedge \pi^*s\wedge d\pi^*\chi_{\epsilon}\wedge \pi^*t$. Since $\pi^*s, \pi^*t$ have nice expressions, for a test form $\xi$ of matching degree, it suffices to consider
\[
\bigg[\lim_{\epsilon\to 0}\chi\left(\frac{|\pi^*F|^2}{\epsilon}\right)^i\frac{d\chi \left(\frac{|\pi^*F|^2}{\epsilon}\right)}{\psi_1}\bigwedge \lim_{\delta\to 0}\chi\left(\frac{|\pi^*F|^2}{\delta}\right)^j\frac{d\chi \left(\frac{|\pi^*F|^2}{\delta}\right)}{\psi_2}\bigg](\xi)
\]
\[
=\lim_{\substack{\epsilon\to 0\\\delta\to 0}}\int_X\frac{d\chi \left(\frac{|\pi^*F|^2}{\epsilon}\right)\wedge d\chi \left(\frac{|\pi^*F|^2}{\delta}\right)\wedge\xi}{\psi_1\psi_2}
\]
and
\[
\bigg[\lim_{\epsilon\to 0}\chi\left(\frac{|\pi^*F|^2}{\epsilon}\right)^{i+j}\frac{\bigwedge^2d\chi \left(\frac{|\pi^*F|^2}{\epsilon}\right)}{\psi_1\psi_2}\bigg](\xi)=\lim_{\epsilon\to 0}\int_X\frac{\bigwedge^2d\chi \left(\frac{|\pi^*F|^2}{\epsilon}\right)\wedge\xi}{\psi_1\psi_2}.
\]
The difference is
\[
\lim_{\substack{\epsilon\to 0\\\delta\to 0\\\tau\to 0}}\int_X\frac{\xi}{\psi_1\psi_2}\wedge\left(d\chi \left(\frac{|\pi^*F|^2}{\epsilon}\right)\bigwedge d\chi \left(\frac{|\pi^*F|^2}{\delta}\right)-\bigwedge ^2 d\chi \left(\frac{|\pi^*F|^2}{\tau}\right)\right)
\]
Note that the term in the parenthesis is nonzero only when $|\pi^*F|^2<(1+\nu)\cdot\min\{\epsilon, \delta, \tau\},$ where $supp(\chi)\subseteq [0, 1+\nu).$ 
By definition of $\chi, ||\nabla\chi||_{L^{\infty}(\mathbb{R})}<\infty,$ and 
\[
\left|\left|d\chi \left(\frac{|\pi^*F|^2}{\epsilon}\right)\right|\right|\leq ||\nabla\chi||_{L^{\infty}(\mathbb{R})}\cdot\left|\left|\nabla\left(\frac{|\pi^*F|^2}{\epsilon}\right)\right|\right|,
\]
so the difference does not exceed (writing $k=||\xi||_{L^{\infty}(X)}\cdot||\nabla\chi||_{L^{\infty}(\mathbb{R})}$ \\$\cdot||\nabla(|\pi^*F|^2)||_{L^{\infty}(Supp(\xi))},$ and denoting the region $\{|\pi^*F|^2<(1+\nu)\cdot\min\{\epsilon, \delta, \tau\}\}$ by $C_{\epsilon, \delta, \tau}$),
\[
k\cdot\lim_{\substack{\epsilon\to 0\\\delta\to 0\\\tau\to 0}}\int_{C_{\epsilon, \delta, \tau}\cap Supp(\xi)}\frac{1}{|\psi_1\psi_2|}\left(\frac 1{\epsilon\delta}-\frac 1{\tau^2}\right).
\]
Note that by Hölder's Inequality, the integrand does not exceed
\[
\sqrt{\int_{Supp(\xi)\cap C_{\epsilon, \delta, \tau}}\frac 1{|\psi_1\psi_2|^2}}\cdot\sqrt{\int_{Supp(\xi)\cap C_{\epsilon, \delta, \tau}}\left|\frac 1{\epsilon\delta}-\frac 1{\tau^2}\right|^2}
\]
\[
\leq\sqrt{\int_{Supp(\xi)}\frac 1{|\psi_1\psi_2|^2}}\cdot\sqrt{\int_{Supp(\xi)\cap C_{\epsilon, \delta. \tau}}\frac 2{|\pi^*F|^2}}.
\]
Since $Supp(\xi)$ is compact, from what we have shown before the first term is finite. Also, the second term goes to $0$ as $\epsilon, \delta$ and $\tau$ go to $0$ simultaneously, since $\mu(Z(|\pi^*F|^2))=0$ and after applying a further sequence of resolutions making the degrees of local coordinates low enough in $|\pi^*F|^2$, just as in the proof of Lemma 5.3, we can follow essentially the same proof as for Equation (3)). Then the statement will then follow from Proposition 5.1.
\end{proof}
\begin{proposition}
$\pi_*\left(\lim_{\epsilon\to 0}\pi^*(\chi_{\epsilon})^id\pi^*\chi_{\epsilon}\wedge \pi^*s\wedge\lim_{\delta\to 0}(\pi^*\chi_{\delta})^jd\pi^*\chi_{\delta}\wedge \pi^*t\right)=\lim_{\epsilon\to 0}\chi_{\epsilon}^id\chi_{\epsilon}\wedge s\wedge\lim_{\delta\to 0}\chi_{\delta}^jd\chi_{\delta}\wedge t$, and we can say the same about the other three currents we are considering. \footnote{i.e. Pushing forward by $\pi$ is the same as taking away all the $\pi^*$ in the expressions.}
\end{proposition}
\begin{proof}
This is clear from the definition.
\end{proof}
\begin{remark}
It is clear from our method that \cite[Theorem 5.1]{LW} also applies to any other characteristic class in de-Rham cohomology.
\end{remark}
\section{Chern Currents in Bott-Chern Cohomology}
\subsection{The Bott-Chern Character}
We first define the double complex of Bott-Chern cohomology classes of the cohesive module $(E^{\bullet}, \mathbb{E}^{''}, h).$\cite{BSW} Letting $d=\partial+\overline{\partial}$ be the de-Rham differential, we have
\begin{definition}
$H_{BC}^{p,q}(X)\equiv \left(\mathcal{A}^{p, q}(X)\cap Ker(d)\right)/\overline{\partial}\partial\mathcal{A}^{p-1, q-1}(X).$
\end{definition}
The \textbf{Bott-Chern character} of $(E^{\bullet}, \mathbb{E}^{''}, h)$ is defined by $ch_{BC}(E^{\bullet},\mathbb{E}^{''}, h)=Trs(exp(-\mathcal{R}_h)).$ By \cite[Lemma 2.20]{HQ}, this defines a class in $H_{BC}(X)$. The Bott-Chern character is also independent of the Hermitian metric $h$ by \cite[Corollary 3.14]{HQ}\cite[Theorem 8.2]{BSW}. Therefore it makes sense to write it as $ch_{BC}(E^{\bullet}, \mathbb{E}^{''}).$ If the complex $(\mathcal{A}^{\bullet}(X, E^{\bullet}), \mathbb{E}^{''}_0)$ is exact, then we have
\begin{theorem}\cite[Theorem 4.21]{HQ}
Let $\mathbb{E}_t^{''}=\sum_{k=0}^{N}t^{\frac{1-k}2}\mathbb{E}^{''}|_{E_k}$ for $t>0$ and $\mathcal{R}_t$ denote the corresponding curvatures. Let $N_H$ be the number operator $h\mapsto kh$ if $h\in\mathcal{A}^{\bullet}(X, E_k)$\cite[Definition 4.18]{HQ}\cite[Definition 2.3]{Bis}, and we have
\[
ch_{BC}(E^{\bullet}, \mathbb{E}^{''})=\partial\overline{\partial}\int_1^{\infty}Trs\left(\frac{N_H\cdot exp(-\mathcal{R}_t)}t\right)dt\Rightarrow ch_{BC}(E^{\bullet}, \mathbb{E}^{''})=0.
\]
\end{theorem}
Now we consider the case where $\mathcal{A}^{\bullet}(X, E^{\bullet})$ is not exact. Again, denote the support of this complex by $Z.$ 
\subsection{Transgression Formulae and Superconnection Currents}
There are Bott-Chern currents representing $ch_{BC}(E^{\bullet}, \mathbb{E}^{''}).$ Denote $Z(F)=X'$ and $\phi:\bigwedge T^*X\rightarrow \bigwedge T^*X, a\mapsto (2\pi i)^{-|a|/2}a.$ Now write $\delta_{X'}\in\mathbb{D}^*(X)$ be the current of integration along $X'.$ Let $\mathscr{HE}$ be the sheaf of cohomology groups of $E^{\bullet}.$ Note that $\forall x\in X$ there is a canonical isomorphism $\mathscr{HE}_x\cong\{y\in E^{\bullet}:\mathbb{E}_0^{''}(y)=0, \mathbb{E}_0^{'}(y)=0\},$ so by \cite[Theorem 1.2]{Bis} it inherits a Hermitian metric from $h.$ Let $\nabla^{\mathscr{HE}}$ be the connection compatible with the inherited Hermitian metric. Now we define the following superconnections:
\begin{enumerate}
\item Fix a $y\in (N_{X/X'})_{\mathbb{R}}$ with $\overline{y}\in\overline{(N_{X/X'})_{\mathbb{R}}},$ define $B=\nabla^{\mathscr{HE}}+\partial_y\mathbb{E}^{''}_0+\partial_{\overline{y}}\mathbb{E}^{'}_0.$
\item For $t>0,$ let $A_t=\nabla^{E^{\bullet}}+\sqrt{t}\mathbb{E}_0.$
\end{enumerate}
Now for any $t>0,$ define currents $\zeta_{E^{\bullet}}(t), \zeta^{'}_{E^{\bullet}}(0)$ and $T(h)$ by 
\[
\zeta_{E^{\bullet}}(t)=\frac 1{\Gamma(t)}\int_0^{\infty}u^{t-1}\left(Trs(N_H\cdot exp(-A_u^2))-\int_{X'}Trs(N_H\cdot exp(-B^2))\cdot\delta_{X'}\right)du
\]
\[
\int_X\mu\zeta^{'}_{E^{\bullet}}(0)=\frac{\partial}{\partial t}|_{t=0}\int_X\mu\zeta_{E^{\bullet}}(t), \text{   $\forall \mu\in\mathcal{A}^{\bullet}(X).$}
\]
\[
T(h)=\phi(\zeta^{'}_{E^{\bullet}}(0)).
\]
Then we have the following representation formula
\begin{theorem}\cite[Theorem 2.5]{BGS} 
\[
ch_{BC}(E^{\bullet}, \mathbb{E}^{''})=\left(\int_{N_{X/X'}}\phi(Trs(exp(-B^2)))\right)\delta_{X'}-\frac{\overline{\partial}\partial}{2\pi i}T(h).
\]
\end{theorem}
By \cite[Theorem 3.2]{Bis}, We know that the wave-front set of $\left(\int_{N_{X/X'}}\phi(Trs(exp(-B^2)))\right)\delta_{X'}$ is contained in $(N_{X/X'})_{\mathbb{R}^{*}},$ and there is the following convergence resembling our previous construction:
\begin{theorem}
As $t\to\infty,$ we have
\[
Trs(exp(-A_t^2))\to\left(\int_{N_{X/X'}}Trs(exp(-B^2))\right)\delta_{X'}
\]
and also (abusing the notation a bit and writing $A_t$ also as the current $\xi\mapsto\int_XTrs(exp(-A_t^2))\wedge\xi$ and $B$ by the current $\xi\mapsto\left(\int_{N_{X/X'}}Trs(exp(-B^2))\right)\delta_{X'}(\xi).$)
\[
\lim_{t\to\infty}\sup_{\xi\in\Gamma}|\xi|^m\cdot\left|\widehat{\phi\cdot(A_t-B)}(\xi)\right|=\lim_{t\to\infty}\sup_{\xi\in\Gamma}|\xi|^m\cdot\left|(A_t-B)(\phi\widehat{\xi})\right|=0,
\]
with the following nice convergence: $\exists C'>0$ such that for $t\ll 1,$ we have
\[
\lim_{t\to\infty}\sup_{\xi\in\Gamma}|\xi|^m\cdot\left|\widehat{\phi\cdot(A_t-B)}(\xi)\right|\leq\frac{C'}{\sqrt{t}},
\]
for any fixed $m\geq 1,$ for any open $U\subseteq X$ biholomorphic to a ball and contained in a trivializing neighborhood of $T_{\mathbb{R}}^{*}X$ and any smooth function $\phi$ supported on $U,$ and any $\Gamma$ a closed cone such that on $U\cap X', \Gamma\cap (N_{X/X'})_{\mathbb{R}}^*=\{0\}.$
\end{theorem}
Observe that all the above constructions involve only the degree-$0$ and degree-$1$ terms of $\mathbb{E}$. It is then natural to ask the following question:
\begin{question}
What effects do the $\mathbb{E}_k^{''}$ terms ($k\geq 2$) have on $ch_{BC}(E^{\bullet}, \mathbb{E}^{''})$?
\end{question}
The answer is that they have no effects. This follows from Qiang's transgression formula with respect to superconnections\cite{HQ}.
\subsubsection{Known Transgression Formulae}
There are two main types of transgression formulae. The first type is with respect to the moduli space (with the topology of uniform $C^{\infty}$ convergence on compact sets) $\mathscr{M}$ of Hemitian metrics on $E^{\bullet}.$\cite[Theorem 8.1.2]{BSW}\cite[Theorems 2.1, 2.2, 2.4]{Bis}\cite[Proposition 3.10, Corollary 3.13, Theorem 3.19]{HQ}. The second type is with respect to the moduli space of superconnections. We will need the following (combining Corollary 4.8 and Proposition 4.15 of \cite{HQ}):
\begin{theorem} Let $f$ be a convergent power series. Let $\mathcal{E}$ be the space of all $\overline{\partial}-$superconnections of degree-$1$ on $(E^{\bullet}, h).$ Then $\exists\delta_1, \delta_2,$ which are $1-$forms on the subspace of $\mathcal{A}^{\bullet, 0}(X, End(E^{\bullet}))$ of exotic degree\footnote{From Proposition 3.1, we can write an element $A\in \mathcal{A}^{p, q}(X, End(E^{\bullet}))$ as $\phi\otimes\tau,$ with $\phi\in\mathcal{A}^{p, q}(X)$ and $\tau\in End^d(E^{\bullet})$ for some $d.$ Then we define the exotic degre to be $d+q-p.$} $-1$ and the subspace of $\mathcal{A}^{0, \bullet}(X, End(E^{\bullet}))$ of exotic degree $1$ respectively, and $\gamma_1, \gamma_2$ which are sections to the subspaces of exotic degree $0$ of $\mathcal{A}^{0, \bullet}(X, End(E^{\bullet}))$ and $\mathcal{A}^{\bullet, 0}(X, End(E^{\bullet}))$ respectively such that
\[
-d^{\mathcal{E}}Trs(f(\mathcal{R}_h))=\partial Trs(f'(\mathcal{R}_h)\cdot\delta_1)+\overline{\partial}Trs(f'(\mathcal{R}_h)\cdot\delta_2)
\]
and
\[
\overline{\partial} Trs(f'(\mathcal{R}_h\cdot\gamma_1))=Trs(f'(\mathcal{R}_h)\cdot\delta_1)
\]
\[
\partial Trs(f'(\mathcal{R}_h\cdot\gamma_2))=Trs(f'(\mathcal{R}_h)\cdot\delta_2).
\]
\end{theorem}
For the construction of $\gamma_1, \gamma_2, \delta_1, \delta_2$ refer to Definition 4.11 and Definition 4.5 in \cite{HQ}. It then directly follows that in $H_{BC}^{\bullet}(X),$
\begin{corollary}
Let $\psi_t(\mathbb{E}^{''})\in\mathcal{E}$ be the superconnection $\mathbb{E}^{''}_0+\mathbb{E}^{''}_1+t\sum_{t\geq 2}\mathbb{E}^{''}_k,$ and let the associated curvature forms be $\phi_t(\mathcal{R}_h).$ Then $Trs(f(\phi_1(\mathcal{R}_h)))\simeq Trs(f(\phi_0(\mathcal{R}_h)))=Trs(f(\mathcal{R}_h))$ in $H_{BC}^{\bullet}(X).$
\end{corollary}
\newpage
\bibliographystyle{abbrv}
\bibliography{Thesis}
\end{document}